\documentclass[12pt]{amsart}
\usepackage{amsfonts}
\usepackage{amssymb}
\usepackage{graphicx}
\usepackage{enumerate}
\usepackage{pgf,tikz}
\usepackage[letterpaper, left=2.5cm, right=2.5cm, top=2.5cm,
bottom=2.5cm,dvips]{geometry}
\usepackage[backref]{hyperref}
\hypersetup{
	colorlinks,
	linkcolor={blue!60!black},
	citecolor={green!60!black},
	urlcolor={red!60!black}
}
\newenvironment{proof1}[1][Proof]{\noindent\textbf{#1} }{\ \rule{0.5em}{0.5em}}

\newtheorem{theorem}{Theorem}
\theoremstyle{plain}

\newtheorem{conjecture}{Conjecture}

\newtheorem{lemma}{Lemma}

\newtheorem{proposition}{Proposition}

\numberwithin{equation}{section}

\begin{document}
\title[Dilworth's Theorem for Borel Posets]{Dilworth's Theorem for Borel Posets}

\author{Bart\l omiej Bosek}
\address{Theoretical Computer Science Department, Faculty of Mathematics and Computer Science, Jagiellonian
University, 30-348 Krak\'{o}w, Poland}
\email{bosek@tcs.uj.edu.pl}
\author{Jaros\l aw Grytczuk}
\address{Faculty of Mathematics and Information Science, Warsaw University
	of Technology, 00-662 Warsaw, Poland}
\email{j.grytczuk@mini.pw.edu.pl}
\author{Zbigniew Lonc}
\address{Faculty of Mathematics and Information Science, Warsaw University
	of Technology, 00-662 Warsaw, Poland}
\email{z.lonc@mini.pw.edu.pl}

\thanks{This work was supported by the European Regional Development Fund under the grant No. POIR.01.01.01-00-0124/17-00 and by FinAi S.A. funding.}

\begin{abstract}
A famous theorem of Dilworth asserts that any finite poset of width $k$ can be decomposed into $k$ chains. We study the following problem: given a Borel poset $P$ of finite width $k$, is it true that it can be decomposed into $k$ Borel chains? We give a positive answer in a special case of Borel posets embeddable into the real line. We also prove a dual theorem for posets whose comparability graphs are locally countable.
\end{abstract}

\maketitle

\section{Introduction}

Dilworth's theorem \cite{Dilworth} is a fundamental result in structural theory of posets (see \cite{Trotter}). It asserts that any finite poset $P$ of width $k$ is decomposable into $k$ chains. Recall that the \emph{width} of a poset $P$ is the maximum size of an antichain in $P$. Our aim is to verify whether analogous statement holds in the realm of Borel posets. A \emph{Borel poset} $P=(V,\preceq)$ is a poset whose comparability relation is a Borel subset of $V\times V$, where $V$ is a standard Borel space (see \cite{HarringtonMS}).

\begin{conjecture}
Every Borel poset $P$ of finite width $k$ is decomposable into $k$ Borel chains.
\end{conjecture}

Notice that by the Compactness Principle, every infinite poset of finite width $k$ can be split into $k$ chains, if no restrictions on set-theoretic properties of these chains are imposed. Actually we do not know at present if there is any finite bound (depending on $k$) on the number of chains in Conjecture 1. However, \emph{thin} Borel posets (no antichain is a perfect set) can be decomposed into a countable number of Borel chains, as proved by Harrington, Merker, and Shelah in \cite{HarringtonMS}.

Our main result confirms Conjecture 1 for a special type of Borel posets that can be embedded into the real line. We call a (not necessarily Borel) poset $P=(V,\preceq)$, where $V\subseteq \mathbb{R}$, {\it realistic} if $x\preceq y$ implies $x\leqslant y$, for every $x,y\in V$. In Section 2 we prove that Conjecture 1 is indeed true for realistic Borel posets $P=(\mathbb{R},\preceq)$  (Theorem \ref{Dilworth}). We derive this fact easily from another result asserting that maximal chains in such posets are Borel (Theorem \ref{max_chain}).

One may have a feeling that realistic Borel poset is a rather restrictive notion. However, as proved by Hladk\'{y}, M\'{a}th\'{e}, Patel, and Pikhurko \cite{HladkyMPP}, every \emph{measurable} poset $P$ (considered as an ordered probability space) can be approximated (in a measurable way) by some realistic poset on the unit interval (with the Lebesgue measure). One may therefore expect that by using tools from \cite{HladkyMPP} at least a measurable version of Conjecture 1 is achievable.

\section{Dilworth's theorem for realistic posets}

The following result is the key element in proving Dilworth's theorem for realistic Borel posets.

\begin{theorem}\label{max_chain} Every maximal chain in a realistic Borel poset $P=(\mathbb{R},\preceq)$ of finite width is Borel.
\end{theorem}

Our main result is an easy consequence of this theorem.

\begin{theorem}\label{Dilworth} Every realistic Borel poset $P=(\mathbb{R},\preceq)$ of finite width $k$ has a partition into $k$ Borel chains.
\end{theorem}
\begin{proof}
	It follows easily from the Compactness Principle that $P$ has a partition into $k$ (not necessarily Borel) chains. We extend these chains to maximal ones, say $C_1,C_2,\ldots,C_k$. The chains are Borel by Theorem \ref{max_chain}. We define $$C_i'=C_i-(C_1\cup C_2\cup\cdots\cup C_{i-1})$$ for $i=1,2,\ldots,k$. Clearly, the chains $C_1',C_2',\ldots,C_k'$ are Borel and form a partition of $P$.
\end{proof}

Before presenting the proof of Theorem \ref{max_chain}, we shall introduce some notation and prove several lemmas.

For a poset $P=(V,\preceq)$ and $x\in V$ we denote by $I_P(x)$ the set of elements of $P$ incomparable with $x$.  For any set $X\subseteq V$, let $I_P(X)$ be the set of elements of $P$ incomparable with some element of $X$, that is,
$$I_P(X)=\bigcup_{x\in X}I_P(x).$$
By $G_P$ we denote the \emph{incomparability graph} of $P$, that is, the graph whose vertices are the elements of $P$ with edges joining pairs of incomparable elements of $P$. For any set $X\subseteq\mathbb{R}$ by $\inf X$ (resp. $\sup X$) we always mean the infimum (resp. supremum) of $X$ in $\mathbb{R}$. Moreover, for $a,b\in \mathbb{R}$, $a<b$, let $[a,b]$ denote the closed interval in $\mathbb{R}$.

Let $P=(V,\preceq)$ be a realistic poset. For a component $C$ of the incomparability graph $G_P$ we define $V_C$ to be the vertex set of $C$. We denote by $I_C$ an interval in $\mathbb{R}$ with the ends $\underline{C}=\inf V_C$ and $\overline{C}=\sup V_C$ such that $\underline{C}\in I_C$ (resp. $\overline{C}\in I_C$) if and only if $\underline{C}\in V_C$  (resp. $\overline{C}\in V_C$).

\begin{lemma}\label{lem0}
	Let $P=(V,\preceq)$ be a realistic poset with incomparability graph $G_P$.
	\begin{enumerate}[(i)]
		\item If $C$ is a component of $G_P$, then $V_C=V\cap I_C$.
		\item The intervals $I_C$ for different components $C$ of $G_P$ are disjoint.
		\item The number of nontrivial (non-singleton) components of $G_P$ is countable.
	\end{enumerate}
\end{lemma}
\begin{proof}
	To prove (i) we observe that for any $x\in V_C$, we have $\underline{C}\leqslant x\leqslant \overline{C}$. So, by the definition of $I_C$, we get $x\in V\cap I_C$. Hence $V_C\subseteq V\cap I_C$.
	
	Suppose now that $x\in V\cap I_C$. By the definition of $I_C$, if $x=\underline{C}$ or $x=\overline{C}$, then $x\in V_C$. Otherwise, there are $a,b\in V_C$ such that $a< x<b$. Let $a=x_1,x_2,\dots,x_k=b$ be a path in $C$. Consider the largest $i$ such that $x_i\leqslant x$. Clearly, $i$ is well-defined and $i<k$. Then, $x_i\leqslant x<x_{i+1}$. Suppose $x$ is comparable in $P$ to both $x_i$ and $x_{i+1}$. Since $x_i\leqslant x<x_{i+1}$, we have $x_i\preceq x\prec x_{i+1}$ because the poset $P$ is realistic. We have got a contradiction because $x_i$ and $x_{i+1}$ are not comparable in $P$ as $x_ix_{i+1}$ is an edge in the graph $G_P$. Thus, either $x$ and $x_i$ or $x$ and $x_{i+1}$ are not comparable in $P$. In both cases $x\in V_C$.
	
	To show (ii) consider two different components $C$ and $C'$ of the graph $G_P$ and suppose that there is $x\in I_C\cap I_{C'}$. We can assume without loss of generality that $\overline{C}\leqslant \overline{C'}$. Then, there is $y\in V_C\subseteq V$ such that $x\leqslant y\leqslant\overline{C}$. Clearly, $y\in I_C\cap I_{C'}$. So, by (i) we get $y\in V_C\cap V_{C'}$, a contradiction.
	
	If a component $C$ is nontrivial, then the interval $I_C$ has positive length. Thus, it contains a rational point. Since by (ii) the intervals $I_C$ are disjoint and the number of rational points on the real line is countable, so is the number of nontrivial components of the graph $G_P$, which proves (iii).
\end{proof}

\begin{lemma}\label{lem1}
	Let $P=(V,\preceq)$ be a realistic poset which has a partition into two disjoint chains $X$ and $Y$. If no element of $X$ is comparable to all elements of $P$, then there are countably many elements $b_1, b_2,\ldots \in Y$ such that
	\[
	X=\bigcup_{i=1}^\infty I_P(b_i).
	\]
\end{lemma}
\begin{proof}
	Since no element of $X$ is comparable to all elements of $P$, all components of the incomparability graph $G_P$ intersecting $X$ are nontrivial (i.e. non-singleton). Consider any such component $C$. Let $V_C$ be the set of vertices of $C$.  Denote $A=X\cap V_C$ and $B=Y\cap V_C$.
	
	We claim that $A$ is a union of countably many sets of the form $I_P(b)$, where $b\in B$.
	
	If $|A|=1$, then $A=\{ a\}\subseteq I_P(b)$, for some $a\in A$ and $b\in B$. Assume now that $|A|>1$. Let $\overline{a}=\sup A$ and $\underline{a}=\inf A$. If $\overline{a}\not\in A$ (resp. $\underline{a}\not\in A$), then define $x_1,x_2,\ldots$ (resp.  $z_1,z_2,\ldots$)  to be an increasing (resp. decreasing) sequence of elements of $A$ convergent to $\overline{a}$ (resp. to $\underline{a}$). If $\overline{a}\in A$ (resp. $\underline{a}\in A$), then $x_1=x_2=\ldots=\overline{a}$ (resp.  $z_1=z_2=\ldots=\underline{a}$). We can assume without loss of generality that $x_1>z_1$. We observe that 
\begin{equation}\label{eq2}
A\subseteq\bigcup_{i=1}^\infty[z_i,x_i].
\end{equation}
For every $i=1,2,\ldots$, consider a path $$z_i=a_1^i,b_1^i,a_2^i,b_2^i,\ldots,a_{t}^i,b_{t}^i,a_{t+1}^i=x_i$$ joining $z_i$ and $x_i$ in $C$. Clearly, $a_1^i,\ldots,a_{t+1}^i\in A$ and $b_1^i,\ldots,b_{t}^i\in B$.
	Let $I_j^i$ be the closed interval with the ends $a_j^i$ and $a_{j+1}^i$, that is, $I_j^i=[a_j^i,a_{j+1}^i]$ if $a_j^i< a_{j+1}^i$, or $I_j^i=[a_{j+1}^i,a_j^i]$ otherwise.
	
	We shall show that 
\begin{equation}\label{eq3}
I_j^i\cap A\subseteq I_P(b_j^i),
\end{equation}
 for $j=1,2,\ldots,t$.  Let us assume that $I_j^i=[a_j^i,a_{j+1}^i]$ (the case $I_j^i=[a_j^i,a_{j+1}^i]$ is analogous). Suppose that for some $c\in [a_j^i,a_{j+1}^i]\cap A$ we have $c\preceq b_j^i$ (respectively, $b_j^i\preceq c$). Then $a_j^i\preceq c\preceq b_j^i$ (resp. $b_j^i\preceq c\preceq a_{j+1}^i$), because $a_j^i,a_{j+1}^i,c$ are pairwise comparable in $P$ as they are elements of the chain $X$ and $P$ is a realistic poset. We have got a contradiction because $a_j^ib_j^i$ (resp. $b_j^ia_{j+1}^i$) is an edge in $G_P$. Thus, indeed, (\ref{eq3}) holds.
 
 Moreover, $$[z_i,x_i]\subseteq\bigcup_{j=1}^tI_j^i,$$ thus, by (\ref{eq3}), 
\[
	[z_i,x_i]\cap A\subseteq\bigcup_{j=1}^tI_j^i\cap A\subseteq \bigcup_{j=1}^tI_P(b_j^i)\subseteq A.
\]
	This inclusion and (\ref{eq2}) imply
	\begin{equation}\label{eq1}
	A=\bigcup_{i=1}^\infty [z_i,x_i]\cap A=\bigcup_{i=1}^\infty\bigcup_{j=1}^tI_P(b_j^i),
	\end{equation}
which completes the proof of the claim.
	
	By Lemma \ref{lem0}(iii), there are countably many nontrivial components in $G_P$. As we have already observed, all components of $G_P$ intersecting $X$ are nontrivial. Denote vertex sets of these components by  $V_{C_1},V_{C_2},\ldots$ and let $A_i=X\cap V_{C_i}$ for every $i=1,2,\ldots$. Since $X=\bigcup_{i=1}^\infty A_i$, the lemma follows by the claim.
\end{proof}

\begin{lemma}\label{lem2}
	Let $P=(V,\preceq)$ be a realistic poset of a finite width. For any chain $Y$ in $P$ there are countably many elements $b_1, b_2,\ldots\in Y$ such that
	\[
	I_P(Y)=\bigcup_{i=1}^\infty I_P(b_i).
	\]
\end{lemma}
\begin{proof}
	By the Compactness Principle the set $I_P(Y)$ is a union of chains, say $X_1,\ldots,X_r$, for some finite $r$.
	
	Let $P_i$ be the ordered set induced in $P$ by the set $X_i\cup Y$, for $i=1,\ldots,r$. Since $X_i\subseteq I_P(Y)$, no element of $X_i$ is comparable to all elements of $P_i$.
	Applying Lemma \ref{lem1} to each ordered set $P_i$ gives
	\[
	X_i=\bigcup_{j=1}^\infty I_{P_i}(b_j^i)\subseteq \bigcup_{j=1}^\infty I_P(b_j^i)\subseteq I_P(Y),
	\]
	for some $b_i^j\in Y$, which completes the proof because $I_P(Y)=X_1\cup\cdots\cup X_r$.
\end{proof}

\begin{proof1}{\bf of Theorem \ref{max_chain}.}
	We observe that for any $x\in \mathbb{R}$, the set $U_P(x)=\{y\in \mathbb{R}: x\preceq y\}$ is a vertical section of the Borel set $E=\{(x,y)\in\mathbb{R}\times \mathbb{R}: x\preceq y\}$, so it is Borel. Similarly, the set $D_P(x)=\{y\in \mathbb{R}: y\preceq x\}$ is a horizontal section of $E=\{(y,x)\in\mathbb{R}\times \mathbb{R}: y\preceq x\}$, so it is Borel, too. Since $$I_P(x)=\mathbb{R}-(U_P(x)\cup D_P(x)),$$ it follows that the set $I_P(x)$ is also Borel.
	
	Let $Y$ be a maximal chain in $P$. Then $Y=\mathbb{R}-I_P(Y)$. Hence $Y$ is a Borel set by the observation in the preceding paragraph and Lemma \ref{lem2}.
\end{proof1}
\medskip

Let us conclude this section with a general observation on realistic posets which follows from Lemma \ref{lem2} and is perhaps interesting by itself.
\begin{theorem}
	Let $P=(V,\preceq)$ be a realistic poset of a finite width. Then every maximal chain $Y$ in $P$ contains a countable chain $X$ such that $Y$ is the only extension of $X$ to a maximal chain in $P$.  
\end{theorem}
\begin{proof}
	Let $X=\{b_1,b_2,\ldots\}$, where $b_1,b_2,\ldots$ are the elements of $Y$ whose existence is guaranteed  by Lemma \ref{lem2}.  Since the chain $Y$ is maximal, $Y=V-I_P(Y)$. Thus,  by Lemma \ref{lem2} we get
	$$
	Y=V-\bigcup_{i=1}^\infty I_P(b_i)=\bigcap_{i=1}^\infty(V-I_P(b_i))=\bigcap_{i=1}^\infty(U_P(b_i)\cup D_P(b_i)).
	$$
	Clearly, every extension of the chain $X$ to a maximal chain is a subset of $$\bigcap_{i=1}^\infty(U_P(b_i)\cup D_P(b_i)),$$
	which completes the proof of the theorem. 
\end{proof}
\section{Dual version of Dilworth's theorem}
Recall that the \emph{height} of a poset $P$ is the largest size of a chain in $P$. A theorem dual to Dilworth's theorem asserts that every poset of finite height $h$ can be decomposed into $h$ antichains.

\begin{conjecture}
	Every Borel poset $P$ of finite height $h$ is decomposable into $h$ Borel antichains.
\end{conjecture}

We do not know if the statement of this conjecture holds even for realistic posets. However, we can prove it for any Borel poset whose comparability graph is locally countable.

We shall use the following Lusin-Novikov Theorem (see Kechris \cite{Kechris}). 
\begin{theorem}
	Let $X$ and $Y$ be standard Borel spaces and let $W\subseteq X\times Y$ be Borel. If every section $W_y=\{ x\in X:(x,y)\in W\}$ is countable, then the projection ${\rm proj}_Y(W)=\{y\in Y: (x,y)\in W\ {\rm for\ some}\ x\in X\}$ is Borel.
\end{theorem}

For a directed graph $G=(V,E)$ and a subset of vertices $A\subseteq V$, let us denote $N^+_G(A)=\{ u\in V: (v,u)\in E \ {\rm for\ some }\ v\in A\}$ and $N^-_G(A)=\{ u\in V: (u,v)\in E \ {\rm for\ some }\ v\in A\}$.

\begin{lemma}\label{lem4}
	Let $G=(V,E)$ be a locally countable directed Borel graph. Then for every Borel set $A\subseteq V$, the set $N^+_G(A)$ is Borel.
\end{lemma}
\begin{proof}
	Clearly, the set $W=(A\times V)\cap E$ is Borel. By local countability of $G$, for any $v\in V$, the set $$W_v=\{ u\in A: (u,v)\in E\}\subseteq N^-_G(v)$$ is countable. It follows from Lusin-Novikov Theorem that the set ${\rm proj}_{V}(W)=N^+_G(A)$ is Borel.
\end{proof}

\begin{proposition}\label{prop1}
	Let $G=(V,E)$ be a locally countable directed acyclic Borel graph in which the longest directed path has $\ell<\infty$ vertices. Then there is a partition of the vertex set of $G$ into $\ell$ independent Borel sets.
\end{proposition}
\begin{proof}
	We apply induction on $\ell$. The statement is trivially true for $\ell=1$. To prove the induction step consider the set $X=\{ v\in V: N^-_G(v)=\emptyset\}$. Clearly, the set $X$ is independent and contains all initial vertices of directed paths of length $\ell$. Observe that $N^+_G(V)=V-X$, so by Lemma \ref{lem4},  $X$ is Borel. By removing $X$ from $G$ we get a graph $G'=(V-X,E')$. This graph is Borel because $E'=E\cap (V-X)^2$. Obviously, longest directed paths in $G'$ have $\ell-1$ vertices, so we are done by the induction hypothesis.
\end{proof}
The above proposition gives immediately the following result.
\begin{theorem}\label{cor1}
	Every Borel poset of finite height $h$ whose comparability graph is locally countable has a partition into $h$ Borel antichains.
\end{theorem}

Recall that projections of Borel sets are measurable. Therefore, we can remove the assumption of local countability in Proposition \ref{prop1} (resp. Theorem \ref{cor1}) and prove the existence of appropriate partition into \emph{measurable} (instead of Borel) independent sets (resp. antichains). In particular we get the following statement.

\begin{theorem}\label{prop2}
	Every Borel poset of finite height $h$ can be decomposed into $h$ measurable antichains.
\end{theorem}

\section{Some remarks}

Decomposition of a poset $P$ into chains is clearly equivalent to a proper coloring of its incomparability graph $G_P$. It is not hard to see that if $P$ is a Borel poset, then $G_P$ is a Borel graph. Hence, Conjecture 1 is equivalent to the statement that every Borel poset $P$ of finite width satisfies
 $$\chi_B(G_P)=\chi(G_P),$$
 where $\chi_B(G)$ denotes the \emph{Borel chromatic number} of a Borel graph $G$, that is, the least number of Borel independent sets covering the vertex set of $G$. The idea of studying graph theoretic concepts in the tilth of Borel spaces was introduced by Kechris, Solecki, and Todorcevic \cite{KechrisST}, and further developed by many researchers (see a survey paper by Kechris and Marks \cite{KechrisMarks})

Consider a locally countable undirected Borel graph $G=(V,E)$. It is not hard to see that components of $G$ have countably many vertices. This graph defines a Borel equivalence relation $E_G$ on the set $V$:
\[
xE_Gy {\rm \ if\ there \ is \ a\ (finite)\ path\ from} \ x\ {\rm to}\ y.
\]
A Borel set $A\subseteq V$  is said to be a \emph{Borel transversal} for an equivalence relation $E$ if $A$ intersects every $E$-class in exactly one point. A Borel equivalence relation $E\subseteq V\times V$ is {\it smooth} if $E$ admits a Borel transversal. The following result was proved by Conley and Miller in \cite{ConleyMiller}.
\begin{theorem}\label{prop3}
	If $G$ is a locally countable Borel graph for which $E_G$ is smooth, then $$\chi_B(G) =  \chi(G).$$
\end{theorem}

This immediately implies the following result.

\begin{theorem}\label{prop3}
	If $P$ is a Borel poset of finite width $k$ whose incomparability graph $G_P$ is locally countable with a smooth relation $E_{G_P}$, then $P$ can be decomposed into $k$ Borel chains.
\end{theorem}

Finally, let us consider a measurable version of Conjecture 1. Assume now that $P=(V,\preceq)$ is a \emph{measurable} poset, that is, an \emph{ordered probability space} with some probabilistic measure $\mu$ (see \cite{HladkyMPP}). A surprising result, conjectured by Janson \cite{Janson} and proved by Hladk\'{y}, M\'{a}th\'{e}, Patel, and Pikhurko \cite{HladkyMPP}, asserts that any atomless measurable poset can be \emph{included} into some realistic (Lebesgue) measurable poset (Theorem 1.10). In view of this and our Theorem 2, the following conjecture seems plausible.

\begin{conjecture}
	Every measurable poset $P$ of finite width $k$ is decomposable (up to a null set) into $k$ measurable chains. In other words, $\chi_{\mu}(G_P)=k$, where $\chi_{\mu}(G)$ denotes the measurable chromatic number of $G$.
\end{conjecture}

\end{document}